\documentclass[a4paper,12pt]{article}
\catcode`\"=13
\def"#1{\v#1}
\usepackage{graphicx,amsmath,amssymb,euscript,amsfonts}
\usepackage{textcomp}
\graphicspath{{images/}}
\usepackage[english]{babel}

\newcommand{\qed}{\hfill\hbox{\rule{3pt}{6pt}}}

\newtheorem{theorem}{Theorem}[section]
\newtheorem{lemma}[theorem]{Lemma}
\newtheorem{corollary}[theorem]{Corollary}
\newtheorem{proposition}[theorem]{Proposition}

\newtheorem{remark}[theorem]{Remark}
\date{}
\begin{document}
\vskip 6cm

\title{\bf On certain topological indices of graphs}
\markright{Abbreviated Article Title}

\author{
  Arber Avdullahu\\
                University of Primorska, Koper, Slovenia \\
         \texttt{arbr.avdullahu@gmail.com}
               \and
                  Slobodan Filipovski\footnote{Supported in part by the Slovenian Research Agency (research program P1-0285 and Young Researchers Grant).} \\
         University of Primorska, Koper, Slovenia \\
         \texttt{slobodan.filipovski@famnit.upr.si}
          }

 \date{}
\maketitle
\begin{abstract} In this paper we give new bounds for a several vertex-based and edge-based topological indices of graphs: Albertson irregularity index, degree variance index, Mostar and the first Zagreb index. Moreover, we give a new upper bound for the energy of graphs through $IRB$-index. Most of our results rely on a well-known characterization of the Laplacian spectral radius.
\end{abstract}

\section{Introduction}
\quad Let $G$ be an undirected graph with $n$ vertices and $m$ edges without loops and multiple edges. The degree of a vertex $v$, denoted $\deg(v)$, is the number of the vertices connected to $v$. Since $\sum_{v\in V(G)} \deg(v) =2m$, average of vertex degrees can be given as $\overline{d}(G)=\frac{2m}{n}.$  A graph is $k$-regular if every degree is equal to $k$. Otherwise, the graph is said to be an irregular graph.
Once we know the average degree of a graph, it is possible to compute more complex measures of the heterogeneity in connectivity across vertices (e.g., the extent to which there is a very big spread between well-connected and not so well-connected vertices in the graph) beyond the simpler measures of range such as the difference between the maximum degree $\Delta$ and the minimum degree $\delta$.

One such measure was proposed by Tom Snijders in \cite{snijders}, called the degree variance of the graph. This vertex-based measure is defined as the average squared deviation between the degree of each vertex and the average degree:
\begin{equation}\label{varr} Var(G)=\frac{\sum_{u\in V(G)}\left(\deg(u)-\overline{d}(G)\right)^{2}}{n}=\frac{\sum_{u\in V(G)}\left(\deg(u)-\frac{2m}{n}\right)^{2}}{n}.
\end{equation}

Note that the degree variance of a regular graph is always zero.\\
Among the oldest and most studied topological indices, there are two classical vertex-degree based topological indices–the first Zagreb index and second
Zagreb index. The Zagreb indices were first introduced by Gutman et al. in \cite{gutman, gutman2}; they present an important molecular descriptor closely correlated with many chemical properties.
The first Zagreb index $M_{1}(G)$ is defined as
\begin{equation} M_{1}(G)=\sum_{v\in V(G)}\deg(v)^{2}.
\end{equation}
A general edge-additive index is defined as the sum, over all edges, of edge effects. These effects can be of various types, but the most common ones are defined in terms of some property of the end-vertices of the considered edge. In many cases the edge contribution represents how similar its end-vertices are.
In 1997, the Albertson irregularity index of a connected graph $G$, introduced by Albertson [1], was defined by
\begin{equation}\label{irr} Irr(G)=\sum_{(u,v)\in E(G)}|\deg(u)-\deg(v)|.
\end{equation}

This index has been of interest to mathematicians, chemists and scientists from related fields due to the
fact that the Albertson irregularity index plays a major role in irregularity measures of graphs \cite{dimitrov,abdo,chen,reti}, predicting the biological activities and properties of chemical compounds in the QSAR/QSPR
modeling and the quantitative characterization of network heterogeneity. Due to their simple computation, the degree-variance $Var(G)$ and the Albertson index $Irr(G)$ belong to the family of the widely used irregularity indices. Another bond-additive index studied in this paper is $IRB$-index, defined as
\begin{equation}\label{irb} IRB(G)=\sum_{e=(u,v)\in E(G)}(\sqrt{\deg(u)}-\sqrt{\deg(v)})^{2}.
\end{equation}

The Mostar index is a recently introduced bond-additive distance-based graph invariant that measures the degree of peripherality of particular edges and of the graph as a whole.
This index was introduced in \cite{mostar}, and independently in \cite{reti}.
Let $e=(u,v)\in E(G).$ Let $n_{e}(u)$ be the number of vertices of $G$ closer to $u$ than to $v$. The Mostar index of $G$ is defined as
\begin{equation}\label{most} Mo(G)=\sum_{e=(u,v)\in E(G)} | n_{e}(u)-n_{e}(v)|.
\end{equation}


In this paper we present an inequality between the degree variance and the Albertson irregularity index, Theorem 2.2.
As a consequence of this result, we improve the well-known lower bound for the first Zagreb index, $M_{1}(G)\geq \frac{4m^{2}}{n}.$
By involving $IRB$ index we provide an upper bound for the energy of graphs, which presents an improvement of the well-known upper bound $\sqrt{2mn},$ Theorem 2.5.

In Section 3 we focus on the Mostar index. In subsection 3.1 we consider bipartite graphs with diameter three; here we derive an upper bound for $Mo(G)$ which depends on the order of the partite sets. We finish our paper with a general upper bound for the Mostar index, in terms of $m$ and $n$, Theorem 3.11.

Almost all results in this paper share the same proving key which comes from a well-known characterization of the Laplacian spectral radius, Lemma 2.1.

\newpage
\section{Laplacian matrices and their applications in estimating certain graph invariants}

\medskip
\quad Most of the results in this paper are based on a well-known upper bound for the spectral radius of the laplacian matrix. Let $G=(V, E)$ be a graph whose vertices are labelled $\{1,2,\ldots, n\}$, and let $L$ be the Laplacian matrix of $G$. The Laplacian matrix of the graph $G$ is defined as follows
\begin{equation*} \label{laplacian}
L_{ij}  = \left\{
 \begin{array}{lc} -1, & \mbox{ if } (i,j)\in E \\
                              0, & \mbox{ if } (i,j)\notin E \; \mbox {and } i\neq j\\
                              -\sum_{k\neq i} L_{ik}, & \mbox{ if } i=j.
                               \end{array}
                               \right.
                               \end{equation*}

It is well-known that $L$ is a positive semidefinite matrix. More about the properties
of the Laplacian matrices  can be found in \cite{meris, mohar}. In our proofs we use the fact that the largest eigenvalue $\lambda_{max}$ of $L$ satisfies $\lambda_{max}\leq n$, that is, $\frac{\lambda_{\max}}{n}\leq 1.$
\\
For a graph $G$ on $n$ vertices we identify a vector $x\in \mathbb{R}^{n}$ with a function $x: V(G)\rightarrow \mathbb{R}.$ The quadratic form defined by $L$ has the following expression
\begin{equation}\label{quadratic} x^{T}Lx=\sum_{(u,v)\in E(G)}(x(u)-x(v))^{2}.
\end{equation}
We also need Fiedler's \cite{eigenvalue} characterization of $\lambda_{max}.$
\begin{lemma} $$\lambda_{max}=2n\max_{x}\frac{\sum_{(u,v)\in E(G)}(x(u)-x(v))^{2}}{\sum_{u\in V(G)}\sum_{v\in V(G)}(x(u)-x(v))^{2}}$$
where $x$ is a nonconstant vector.
\end{lemma}

\subsection{An inequality between the degree variance and the irregularity index of graphs}

\begin{theorem} Let $G$ be a connected graph on $n$ vertices and $m$ edges. Then
\begin{equation}\label{teorema}Var(G)\geq \frac{Irr^{2}(G)}{mn^{2}}.
\end{equation}
The equality holds if and only if $G$ is a regular graph.
\end{theorem}
\begin{proof}
From the inequality between arithmetic and quadratic mean for the numbers $|\deg(u)-\deg(v)|$, where $(u,v)\in E(G),$ we get
\begin{equation}
\begin{gathered}
 Irr(G)=\sum_{(u,v)\in E(G)}|\deg(u)-\deg(v)|\leq m\cdot \sqrt{\frac{\sum_{(u,v)\in E(G)}(\deg(u)-\deg(v))^{2}}{m}}\\
=\sqrt{m}\sqrt{\sum_{(u,v)\in E(G)}(\deg(u)-\deg(v))^{2}}= \sqrt{m}\sqrt{\sum_{(u,v)\in E(G)}((\deg(u)-\frac{2m}{n})-(\deg(v)-\frac{2m}{n}))^{2}}.
\end{gathered}\label{hm}
\end{equation}

For the graph $G$ we define a vector $x \in \mathbb{R}^{n}$ such that $x(u)=\deg(u)-\frac{2m}{n},$ where $u\in V(G).$
The inequality in (\ref{hm}) becomes
\begin{gather*}
Irr(G)\leq \sqrt{m} \sqrt{\sum_{(u,v)\in E(G)}\left((\deg(u)-\frac{2m}{n})-(\deg(v)-\frac{2m}{n})\right)^{2}}=
\sqrt{m}\sqrt{\sum_{(u,v)\in E(G)}(x(u)-x(v))^{2}}.
\end{gather*}
Using arithmetic calculations it is easy to see that:
\begin{equation}\label{eq:1}
\frac{1}{2}\cdot \sum_{u\in V(G)}\sum_{v\in V(G)}(x(u)-x(v))^{2} = n\cdot \sum_{u\in V(G)}x(u)^{2}-(\sum_{v\in V(G)}x(v))^{2}
\end{equation}
Now, applying Lemma 2.1 and equation (\ref{eq:1}) we obtain
\begin{gather*}
Irr(G)\leq \sqrt{m}\sqrt{\frac{\lambda_{\max}}{2n}\cdot \sum_{u\in V(G)}\sum_{v\in V(G)}(x(u)-x(v))^{2}}=
\sqrt{m}\sqrt{\frac{\lambda_{\max}}{n}[n\sum_{u\in V(G)}x(u)^{2}-(\sum_{v\in V(G)}x(v))^{2}]}.
\end{gather*}
From $\sum_{u\in V(G)}x(u)^{2}=\sum_{u\in V(G)} (\deg(u)-\frac{2m}{n})^{2}=n\cdot Var(G)$ and from $\sum_{v\in V(G)}x(v)=\sum_{v\in V(G)}(\deg(v)-\frac{2m}{n})=\sum_{v\in V(G)}(\deg(v)) - 2m=0$ we get

\begin{equation}\label{ravenka} Irr(G)\leq\sqrt{m}\sqrt{\frac{\lambda_{\max}}{n}\cdot n^{2}Var(G)}\leq n\sqrt{m}\sqrt{Var(G)}.
\end{equation}
From (\ref{ravenka}) we get
$$Var(G)\geq \frac{Irr^{2}(G)}{mn^{2}}.$$
Since the equality between arithmetic and quadratic mean holds when $\deg(u)=\deg(v)$ for each $u,v \in V(G)$, we obtain that
the equality in (\ref{teorema}) holds when $G$ is a regular graph. On the other hand, if $G$ is a regular graph we get $Var(G)=Irr(G)=0.$
\qed

\end{proof}

\subsection{A new lower bound for the first Zagreb index}
\quad As a consequence of  Theorem 2.2, we  improve the well-known lower bound for the first Zagreb index, $M_{1}(G)\geq \frac{4m^{2}}{n}$.
\begin{corollary} Let $G$ be a graph with $m$ edges and $n$ vertices. Then
$$M_{1}(G)\geq \frac{4m^{2}}{n}+\frac{Irr^{2}(G)}{mn}.$$
\end{corollary}
\begin{proof} Recall, the first Zagreb index for the graph $G$ is defined as
\begin{equation}\label{eden} M_{1}(G)=\sum_{u\in V(G)} \deg(u)^{2}.\end{equation}
It is easy to show that
\begin{equation}\label{dva}Var(G)=\frac{M_{1}(G)}{n}-\frac{4m^{2}}{n^{2}}.\end{equation}
From Theorem 2.2 and (\ref{dva}) we obtain the required inequality.
\qed
\end{proof}
\begin{remark} Let $G$ be a graph on $n$ vertices and $m$ edges, and let $d_{1},d_{2},\ldots, d_{n}$ be its vertex-degrees. Corollary 2.3 is equivalent to the inequality
$$d_{1}^{2}+d_{2}^{2}+\ldots+d_{n}^{2}\geq \frac{(d_{1}+d_{2}+\ldots+d_{n})^{2}}{n}+\frac{(\sum_{i\sim j}|d_{i}-d_{j}|)^{2}}{mn},$$
which presents an improvement of the inequality between quadratic and arithmetic means for the positive numbers $d_{1},d_{2},\ldots, d_{n}$.
\end{remark}

\subsection{A new upper bound for the energy of graphs}
\quad An \emph{adjacency matrix} $A=A(G)$ of the graph $G$ is the
$n\times n$ matrix
$[a_{ij}]$ with $a_{ij}=1$ if $v_{i}$ is adjacent to $v_{j},$ and $a_{ij}=0$ otherwise.
The eigenvalues $\lambda_{1}, \lambda_{2}, \ldots , \lambda_{n}$ of the graph $G$ are the eigenvalues of its adjacency
matrix $A$. Since $A$ is a symmetric matrix
with zero trace, these eigenvalues are real and their sum is equal to zero.

The \emph{energy }of $G$, denoted by $E(G)$, was first defined by I. Gutman
in \cite{ivan0} as the sum of the absolute values
of its eigenvalues. Thus,
\begin{equation}\label{energija} E(G)=\sum_{i=1}^{n} |\lambda_{i}|.
\end{equation}
This concept arose in theoretical chemistry, since
it can be used to approximate the total $\pi$-electron energy of a molecule.
The first result relating the
energy of a graph with its order and size is the following upper bound
obtained  in 1971 by McClelland \cite{pi}:
\begin{equation}\label{bound2}
E(G)\leq \sqrt{2mn}.
\end{equation}
Since then, numerous other bounds for $E(G)$ were discovered, see \cite{babic, ivan5, jahan1, li}.
In \cite{kulen1} Koolen and Moulton improved the bound (\ref{bound2}) as follows: If $2m>n$ and $G$ is a graph with $n$ vertices and $m$ edges, then
\begin{equation}\label{kulen}
E(G)\leq \frac{2m}{n}+\sqrt{(n-1)\left(2m-\left(\frac{2m}{n}\right)^{2}\right)}.
\end{equation}

We improve the bound in (\ref{bound2}) by using $IRB$ index for a given graph and the technique presented in this paper. Recall, the $IRB$-index for the graph $G$ is defined as \\ $IRB(G)=\sum_{(u,v)\in E(G)}(\sqrt{\deg(u)}-\sqrt{\deg(v)})^2$.

\begin{theorem} Let $G$ be a connected graph on $n$ vertices and $m$ edges. Then
$$E(G) \leq \sqrt{2mn - IRB(G)}.$$
\end{theorem}
\begin{proof}
For the graph $G$ we define a vector $x \in \mathbb{R}^{n}$ such that $x(u)=\sqrt{\deg(u)},$ where $u\in V(G).$ Then we rewrite $IRB(G) = \sum_{(u,v)\in E(G)}(\sqrt{\deg(u)}-\sqrt{\deg(v)})^2=\sum_{(u,v)\in E(G)} (x(u)-x(v))^{2}.$
Applying Lemma 2.1 and equation (8) we obtain
\begin{gather*}\label{energija}
IRB(G)\leq \frac{\lambda_{\max}}{n}\cdot( n\sum_{u\in V(G)}x(u)^{2}-(\sum_{v\in V(G)}x(v))^{2})  \\
\leq  n\sum_{u\in V(G)}x(u)^{2}-(\sum_{v\in V(G)}x(v))^{2} = n\sum_{u\in V(G)}\deg(u)-(\sum_{v\in V(G)}\sqrt{\deg(v)})^{2}.
\end{gather*}
In \cite{ariz} was proven that $E(G)\leq \sum_{u\in V(G)} \sqrt{\deg(u)}.$
Using this estimation we get
$$E(G) \leq \sqrt{2mn - IRB(G)}.$$
\qed
\end{proof}

Unfortunately we are not able to compare the bounds in Theorem 2.5 and (\ref{kulen}).

\section{Mostar index}
\quad As we mentioned in the introduction, the Mostar index is a new bond-additive structural invariant which measures the peripherality in graphs. Moreover, the Mostar index can be used to measure how much a given graph $G$ deviates from being distance-balanced. This index was introduced by Došlić et. al in \cite{mostar} and is studied in several publications, for example, see in \cite{mostar,gao,reti}.
In \cite{mostar} was proven that $Mo(P_{n})\leq Mo(T_{n})\leq (n-1)(n-2)=Mo(S_{n}),$ where $P_{n}, T_{n}$ and $S_{n}$ are paths, trees and stars on $n$ vertices, respectively.
We list several known results.
\begin{proposition}\cite{gao} Let $G$ be a graph of diameter $2$. Then $Mo(G)=Irr(G).$
\end{proposition}
\begin{proposition}\cite{gao} Let $T$ be a tree. Then $Mo(T)$ and $Irr(T)$ have the same parity.
\end{proposition}
\begin{theorem}\cite{gao} Let $T_{n}$ be a tree on $n$ vertices. Then
$$Mo(T_{n})\geq Irr(T_{n})$$
with equality if and only if $T_{n}$ is isomorphic with $S_{n}$.
\end{theorem}
\begin{theorem} \label{retii}\cite{reti} If $G$ is a connected graph of order $n\geq 3$ and size $m$, then
$$0\leq Mo(G)\leq m(n-2)$$
with the left equality if and only if $G$ is a distance-balanced graph and with the right equality if and only if $G$ is isomorphic to the star graph $S_{n}.$
\end{theorem}
In this section we derive several new upper bounds for the Mostar index.
\subsection{Bipartite graphs with diameter three}
\quad In \cite{gao} was proven that the graphs with diameter two have equal Mostar and Albertson index, that is, $Mo(G)=Irr(G).$
In this subsection we consider bipartite graphs with diameter three. Let $G=(V, E)$ be a bipartite graph with diameter three,  of order $n$ and size $m$. Let $V_{1}$ and $V_{2}$ be the partite sets of $G$, that is, $V=V_{1}\cup V_{2}.$
We suppose that $|V_{1}|=n_{1}$ and $|V_{2}|=n_{2}.$  Let $e=(u,v)\in E(G)$ such that $u\in V_{1}$ and $v\in V_{2}.$ Since the diameter of $G$ is three we easily observe that $n_{e}(u)=n_{1}+\deg(u)-\deg(v)$ and $n_{e}(v)=n_{2}+\deg(v)-\deg(u).$ Thus
\begin{equation} \label{mostar} Mo(G)=\sum_{(u,v)\in E(G)}|n_{e}(u)-n_{e}(v)|=\sum_{(u,v)\in E(G)}|(n_{1}+2\deg(u))-(n_{2}+2\deg(v))|.
\end{equation}

Using Lemma 2.1 we derive the following result.

\begin{theorem} Let $G$ be a bipartite graph on $n$ vertices, $m$ edges and  with diameter three. Let $V_{1}$ and $V_{2}$ be the partitive sets of $G$ such that $|V_{1}|=n_{1}$ and $|V_{2}|=n_{2}.$   Then
$$Mo(G)\leq \sqrt{m}\sqrt{\frac{\lambda_{\max}}{n}(n_{1}n_{2}n^{2}-4mn^{2}+4M_{1}(G)n-4n_{1}^{2}n_{2}^{2}-16m^{2}+16n_{1}n_{2}m)}.$$
\end{theorem}
\begin{proof} For the graph $G$ we define a vector $x \in \mathbb{R}^{n}$ as follows:
 \begin{equation*} \label{laplacian}
x(w)  = \left\{
 \begin{array}{lc} n_{1}+2\deg(w), & \mbox{ if } w \in V_{1} \\
                              n_{2}+2\deg(w), & \mbox{ if } w \in V_{2} .\\

                                                             \end{array}
                               \right.
                               \end{equation*}
From (\ref{mostar}) and from the inequality between arithmetic and quadratic mean we get
$$Mo(G)\leq \sqrt{m} \sqrt{\sum_{(u,v)\in E(G)}((n_{1}+2\deg(u))-(n_{2}+2\deg(v)))^{2}}=\sqrt{m}\sqrt{\sum_{(u,v)\in E(G)}(x(u)-x(v))^{2}}.$$
Now, applying Lemma 2.1 we obtain
$$Mo(G)\leq \sqrt{m}\sqrt{\frac{\lambda_{\max}}{2n}\cdot \sum_{u\in V(G)}\sum_{v\in V(G)}(x(u)-x(v))^{2}}=\sqrt{m}\sqrt{\frac{\lambda_{\max}}{n}[n\sum_{u\in V(G)}x(u)^{2}-(\sum_{v\in V(G)}x(v))^{2}]}=$$
$$=\sqrt{m}\sqrt{\frac{\lambda_{\max}}{n}\left(n(n_{1}^{3}+n_{2}^{3}+4M_{1}(G)+4n_{1}m+4n_{2}m)-(n^{2}-2n_{1}n_{2}+4m)^{2}\right)}=$$
$$\sqrt{m}\sqrt{\frac{\lambda_{\max}}{n}(n_{1}n_{2}n^{2}-4mn^{2}+4M_{1}(G)n-4n_{1}^{2}n_{2}^{2}-16m^{2}+16n_{1}n_{2}m)}.$$

\qed

\end{proof}

\begin{remark} If $n_{1}=n_{2}$, then $Mo(G)=2\sum_{(u,v)\in E(G)}| \deg(u)-\deg(v)|=2Irr(G).$ Setting $n_{1}=n_{2}=\frac{n}{2}$ in Theorem 3.5 we get
$$Mo(G)=2\cdot Irr(G)\leq 2\sqrt{m(nM_{1}(G)-4m^{2})\frac{\lambda_{\max}}{n}},$$
which matches with the Goldberg's bound given in \cite{gold}.
\end{remark}

\begin{corollary} Let $G$ be a bipartite graph with diameter three such that $|V_{1}|=n_{1}$ and $|V_{2}|=n_{2}$. Then
$$Mo(G)\leq \sqrt{\frac{n_{1}^{2}n_{2}^{2}n^{2}}{3}+\left(\frac{2n_{1}^{2}n_{2}^{2}}{27}+\frac{n_{1}n_{2}n^{2}}{18}\right)\sqrt{4n_{1}^{2}n_{2}^{2}+3n_{1}n_{2}n^{2}}-\frac{4n_{1}^{3}n_{2}^{3}}{27}}.$$
\end{corollary}
\begin{proof} Clearly, $G$ is a triangle-free graph, and for its first Zagreb index holds $M_{1}(G)\leq mn.$ Using this bound in Theorem 3.5 we get
$$Mo(G)\leq \sqrt{m(n_{1}n_{2}n^{2}-4n_{1}^{2}n_{2}^{2}-16m^{2}+16n_{1}n_{2}m)}.$$
Let $f(x)=-16x^{3}+16n_{1}n_{2}x^{2}+(n_{1}n_{2}n^{2}-4n_{1}^{2}n_{2}^{2})x$ be a function defined on the interval $[1, n_{1}n_{2}].$ Clearly $m\in [1, n_{1}n_{2}].$
Since $f^{'}(x)=-48x^{2}+32n_{1}n_{2}x+n_{1}n_{2}n^{2}-4n_{1}^{2}n_{2}^{2}$ we get
\begin{equation} \label{max}f^{'}(x)\geq 0 \Leftrightarrow  x\in [\frac{n_{1}n_{2}}{3}-\frac{\sqrt{4n_{1}^{2}n_{2}^{2}+3n_{1}n_{2}n^{2}}}{12}, \frac{n_{1}n_{2}}{3}+\frac{\sqrt{4n_{1}^{2}n_{2}^{2}+3n_{1}n_{2}n^{2}}}{12}].
\end{equation}
From (\ref{max}) we conclude that $f(x)$ achieves its maximum (on $[1,n_{1}n_{2}])$ at $m_{\max}=\frac{n_{1}n_{2}}{3}+\frac{\sqrt{4n_{1}^{2}n_{2}^{2}+3n_{1}n_{2}n^{2}}}{12}.$ Hence
$$Mo(G)\leq \sqrt{f(m)}\leq \sqrt{f(m_{\max})}=\sqrt{m_{\max}(n_{1}n_{2}n^{2}-4n_{1}^{2}n_{2}^{2}-16m_{\max}^{2}+16n_{1}n_{2}m_{\max})}=$$
$$=\sqrt{\frac{n_{1}^{2}n_{2}^{2}n^{2}}{3}+\left(\frac{2n_{1}^{2}n_{2}^{2}}{27}+\frac{n_{1}n_{2}n^{2}}{18}\right)\sqrt{4n_{1}^{2}n_{2}^{2}+3n_{1}n_{2}n^{2}}-\frac{4n_{1}^{3}n_{2}^{3}}{27}}.$$

\qed
\end{proof}

In chemical graph theory, the Szeged index is another edge-based topological index of molecule. This index was introduced by Gutman in \cite{gutman10}.  The Szeged index of a connected graph $G$ is defined as
\begin{equation}\label{seged}
Sz(G)=\sum_{e=(u,v)\in E(G)} n_{e}(u)\cdot n_{e}(v).
\end{equation}
In the next result we give a relation between Mostar and Szeged index for bipartite graphs.
\begin{proposition}Let $G$ be a bipartite graph on $n$ vertices and $m$ edges. Then
\begin{equation}\label{bound}Mo(G)\leq \sqrt{m^{2}n^{2}-4mSz(G)}.
\end{equation}
\end{proposition}
\begin{proof} Since $G$ is a bipartite graph, then $n_{e}(u)+n_{e}(v)=n.$ Thus
\begin{equation}\label{haha}
(n_{e}(u)-n_{e}(v))^{2}=n^{2}-4n_{e}(u)n_{e}(v).
\end{equation}
From the inequality between quadratic and arithmetic mean for the numbers $|n_{e}(u)-n_{v}(v)|$ we get

$$\frac{Mo^{2}(G)}{m}=\frac{\left(\sum_{e\in E(G)}|n_{e}(u)-n_{e}(v)|\right)^{2}}{m}\leq \sum_{e\in E(G)}(n_{e}(u)-n_{e}(v))^{2}=mn^{2}-4Sz(G).$$
Thus $Mo^{2}(G)\leq m^{2}n^{2}-4mSz(G).$
\qed
\end{proof}

\begin{remark} When $n_{e}(u)+n_{e}(v)=n$ the parameter $n_{e}(u)n_{e}(v)$ achieves minimum $n-1$ (if one number is $n-1$ and other $1$).
Thus $Sz(G)\geq m(n-1).$ Replacing this in (\ref{bound}) we get
$Mo^{2}(G)\leq m^{2}n^{2}-4mSz(G)\leq m^{2}n^{2}-4m\cdot m\cdot(n-1)=m^{2}(n-2)^{2}.$
Hence
$$Mo(G)\leq m(n-2)$$
which is a trivial upper bound for $Mo(G).$ We note that the equality holds for the complete bipartite graph $K_{1,n-1}.$
\qed
\end{remark}

\subsection{A new upper bound for the Mostar index}
\quad In the last part of the paper we consider non triangle-free graphs. For an edge $(u,v)$ of $G$ we define $n_{uv}=|\{w\;|\; d(w,v)=d(w,u)\}|.$ It can be noted that $0\leq n_{uv}\leq n-2.$ It is easy to see that for any edge $(u,v)$ of $G$, $n_{e}(u)+n_{e}(v)=n-n_{uv}.$
\\Denote by $t(G)$ the number of triangles of a graph $G$.
We will use the following result:
\begin{proposition} For the number of triangles of $G$ it holds
$$\frac{m(4m-n^{2})}{n}\leq 3t(G)\leq \sum_{(u,v)} n_{uv}.$$
\end{proposition}
The first inequality is due to Bollob\' {a}s \cite{bolobas}. Since $n_{uv}$ counts cycles of odd length containing an edge $(u,v)$, the second inequality is obvious.

\begin{theorem} Let $G$ be a graph on $n$ vertices and $m$ edges. Then
$$Mo(G)< \sqrt{m^{2}(n-2)^{2}-\frac{m^{2}(n-2)(4m-n^{2})}{n}}.$$
\end{theorem}
\begin{proof} We already show that $ Mo^{2}(G)\leq m \sum_{(u,v)\in E(G)}(n_{e}(u)-n_{e}(v))^{2}.$ Thus
$$Mo^{2}(G)\leq m\sum_{(u,v)\in E(G)}(n_{e}(u)+n_{e}(v))^{2}-4m\sum_{(u,v)\in E(G)}n_{e}(u)n_{e}(v)\leq$$
$$\leq m\sum _{(u,v)\in E(G)}(n-n_{uv})^{2}-4m \sum_{(u,v)\in E(G)} (n-1-n_{uv})$$
$$=m^{2}n^{2}-4m^{2}(n-1)-m\sum_{(u,v)\in E(G)}n_{uv}(2n-n_{uv}-4)$$
$$\leq m^{2}(n-2)^{2}-m\sum_{(u,v)\in E(G)}n_{uv}(2n-4-(n-2))=$$
$$=m^{2}(n-2)^{2}-m(n-2)\sum_{(u,v)\in E(G)}n_{uv}$$
$$\leq m^{2}(n-2)^{2}-\frac{m^{2}(n-2)(4m-n^{2})}{n}.$$
\qed

\end{proof}
\begin{remark} In the above theorem we can assume that $m > \frac{n{2}}{4}.$ According to the Mantel theorem, it directly implies that $G$ is not a triangle-free graph. This assumption makes our bound better than the existing trivial bound $m(n-2)$ given in Theorem 3.4.
\end{remark}


\section*{Acknowledgment}

The authors would like to thank Dr. Ademir Hujdurović for introducing us with the Mostar index.

\end{document}